\documentclass[12pt]{amsart}
\usepackage{longtable}
\usepackage{amsfonts,amsmath,mathrsfs,amsthm,amscd,amssymb,latexsym,amsbsy,pb-diagram,cite,caption,url}

\usepackage{tikz}
\usetikzlibrary{positioning}

\newtheorem{theorem}{Theorem}[section]

\newtheorem{lemma}[theorem]{Lemma}
\newtheorem{proposition}[theorem]{Proposition}

\theoremstyle{definition}
\newtheorem{definition}[theorem]{Definition}
\newtheorem{procedure}[theorem]{Procedure}
\newtheorem{example}[theorem]{Example}
\newtheorem{remark}[theorem]{Remark}

\numberwithin{equation}{section}

\begin{document}

\title[Minimum spanning paths and Hausdorff distance]{Minimum spanning paths and Hausdorff distance in finite ultrametric spaces}

\author{Evgeniy Petrov}
\address{Institute of Applied Mathematics and Mechanics of the NAS of Ukraine, Dobrovolskogo str. 1, Slovyansk 84100, Ukraine}
\email{eugeniy.petrov@gmail.com}

\begin{abstract}
It is shown that a minimum weight spanning tree of a finite ultrametric space can be always found in the form of path. As a canonical representing tree such path uniquely defines the whole space and, moreover, it has much more simple structure. Thus, minimum spanning paths are a convenient tool for studying finite ultrametric spaces.
To demonstrate this we use them for characterization of some known classes of ultrametric spaces. The explicit formula for Hausdorff distance in finite ultrametric spaces is also found. Moreover, the possibility of using minimum spanning paths for finding this distance is shown.
\end{abstract}



\maketitle

\section{Introduction}
The correspondence between ultrametric spaces and trees is well-known. A convenient presentation of finite ultrametric spaces by monotone rooted trees was proposed in~\cite{GV12,GV}. Such  trees are known as canonical trees. So-called Cartesian trees are an alternative approach, see~\cite{DLW09} and also~\cite[p.~1744]{GV12} for details. These two types of trees can be characterized as labeled trees since the distance between two points in ultrametric space is defined by some label assigned to a vertex of a corresponding representing tree. It is also possible to distinguish the another type of trees so-called weighted trees.  The first example is equidistant trees, see e.g.,~\cite[p.~151]{SS03}. The distance between two points of a space here is defined as a sum of edge weights  belonging to a path connecting these points in the corresponding tree. The second example is well-known minimum weight spanning trees. In this case the distance between two point is defined as a maximal weight of the edge belonging to the path connecting these points. For connections between infinite ultrametric spaces and tress see, e.g.,~\cite{H04, H12, LB17, L03}.

The problem of finding a minimum spanning tree of a weighted graph is well studied, see e.g., \cite{FT87, KKT95, CLRS09, PR02} for  different algorithms. The weighted graph obtained from a finite metric space has a restriction on weights caused by the triangle inequality. The algorithm for this case was considered in~\cite{I99}. The algorithms for construction of minimum spanning trees of finite ultrametric spaces can be found, e.g., in~\cite{GV12, WCT99, CC04, S20}.

A simple algorithm which produces minimum spanning paths in the case of finite ultrametric spaces was proposed in~\cite{D82}.
In Section~\ref{MSPA} we analyse the work of this algorithm on representing trees. A question how to distinguish the set of balls in a finite ultrametric space using only any minimum spanning path of this space is also considered.

In section~\ref{MSTFSC} using minimum spanning paths we give a set of criteria which guarantee that a finite ultrametric space belongs to some special class.  Moreover, we give a criterion which guarantees  that every minimum spanning tree of the space is a path.

Some questions related to Hausdorff metric and Gromov-Hausdorff metric (see, e.g.,~\cite{BBI01,G01} for the definitions of these metrics) in non-Archimedean spaces were considered in~\cite{Q09,Q14}. In particular, in~\cite[Lemma~2.1]{Q14} it was found an explicit formula for Hausdorff distance between any two distinct balls of a finite ultrametric space. See also Lemma 2.6 in~\cite{D19} for a variation of this formula. In Section~\ref{HDandMST} we generalize this result by giving an explicit formula for Hausdorff distance between arbitrary subsets of a finite ultrametric space. We also give an example of using minimum spanning paths for calculating this distance. Note that Hausdorff distance has various applications in both applied and abstract areas of mathematics including computer vision, computer graphics, nonsmooth analysis, optimization theory and calculus of variations, see, e.g.,~\cite{KMCM20, DFS17, AVZ16} and references therein.

Let us introduce main definitions and required technical results.

\begin{definition}\label{d1.1}
An \textit{ultrametric} on a set $X$ is a function $d\colon X\times X\rightarrow \mathbb{R}^+$, $\mathbb R^+ = [0,\infty)$, such that for all $x,y,z \in X$:
\begin{itemize}
\item [(i)] $d(x,y)=d(y,x)$,
\item [(ii)] $(d(x,y)=0)\Leftrightarrow (x=y)$,
\item [(iii)] $d(x,y)\leq \max \{d(x,z),d(z,y)\}$.
\end{itemize}
\end{definition}
Inequality (iii) is often called the {\it strong triangle inequality}. The \emph{spectrum} of an ultrametric space $(X,d)$ is the set
$$
\operatorname{Sp}(X) := \{d(x,y)\colon x,y \in X\}.
$$
The quantity
$$
\operatorname{diam} X := \sup\{d(x,y)\colon x,y\in X\}
$$
is the diameter of $(X,d)$.

A \textit{graph} is a pair $(V,E)$ consisting of a nonempty set $V$ and a set $E$ whose elements are unordered pairs of different points from $V$. For a graph $G=(V,E)$, the sets $V=V(G)$ and $E=E(G)$ are called \textit{the set of vertices (or nodes)} and \textit{the set of edges}, respectively. If $\{x,y\} \in E(G)$, then the vertices $x$ and $y$ are \emph{adjacent}. A graph $G=(V,E)$ together with a function $w\colon E\rightarrow \mathbb{R}^+$ is called a \textit{weighted} graph, and $w$ is called a \textit{weight}.

Recall that a \emph{path} is a nonempty graph $P$ whose vertices can be numbered so that
$$
V(P) = \{x_1,...,x_k\}\quad \text{and} \quad E(P) = \{\{x_1,x_2\},...,\{x_{k-1},x_k\}\}.
$$
A finite graph $C$ is a \textit{cycle} if $|V(C)|\geq 3$ and there exists an enumeration $(v_1,v_2,...,v_n)$
of its vertices such that
\begin{equation*}
(\{v_i,v_j\}\in E(C))\Leftrightarrow (|i-j|=1\quad \mbox{or}\quad |i-j|=n-1).
\end{equation*}
A graph $H$ is a \emph{subgraph} of a graph $G$ if
$$
V(H) \subseteq V(G) \quad \text{and} \quad E(H) \subseteq E(G).
$$
We write $H\subseteq G$ if $H$ is a subgraph of $G$. A cycle $C$ is a \emph{Hamilton cycle} in a graph $G$ if $C\subseteq G$ and $V(C) =V(G)$.

A connected graph without cycles is called a \emph{tree}. A vertex $v$ of a tree $T$ is a \emph{leaf} if the \emph{degree} of $v$ is less than two, i.e.,
$$
\delta(v) = |\{u \in V(T) \colon \{u, v\} \in E(T)\}| < 2.
$$
If a vertex $v \in V(T)$ is not a leaf of $T$, then we say that $v$ is an \emph{internal node} of $T$. A tree $T$ may have a distinguished vertex $r$ called the \emph{root}; in this case $T$ is a \emph{rooted tree} and we write $T=T(r)$.
A \emph{labeled rooted tree} $T=T(r,l)$ is a rooted tree $T(r)$ with a labeling $l\colon V(T)\to L$ where $L$ is a set of labels. Generally we follow terminology used in~\cite{BM}.

\begin{definition}\label{def3.1}
A graph $G$, $E(G) \neq \varnothing$, is called \emph{complete $k$-partite} if its vertices can be divided into disjoint nonempty sets $X_1$, $\ldots$, $X_k$ so that $k \geqslant 2$ and there are no edges joining the vertices of the same set $X_i$ and any two vertices from different $X_i$ and $X_j$, $1\leqslant i,j \leqslant k$ are adjacent. In this case we write $G=G[X_1,\ldots, X_k]$.
\end{definition}
We shall say that $G$ is a {\it complete multipartite graph} if there exists $k$ such that $G$ is complete $k$-partite.

\begin{definition}\label{d14}
Let $(X,d)$ be a finite ultrametric space and let $t\in \operatorname{Sp}(X)$ be nonzero. Denote by $G_{t,X}$ a graph for which $V(G_{t,X})=X$ and
$$
(\{u,v\}\in E(G_{t,X}))\Leftrightarrow (d(u,v)=t).
$$
In particular, for $|X| \geqslant 2$ we write $G_{D, X} := G_{\operatorname{diam} X, X}$ for short and call $G_{D, X}$ the \emph{diametrical graph} of $X$.
\end{definition}

\begin{theorem}[\!\cite{DDP(P-adic)}]\label{t13}
Let $(X,d)$ be a finite ultrametric space, $|X|\geqslant 2$. Then $G_{D,X}$ is complete multipartite.
\end{theorem}

In the following procedure with every nonempty finite ultrametric space $(X, d)$ we associate a labeled rooted tree $T_X=T_X(r,l)$ with $r=X$ and $l\colon V(T_{X})\to \mathbb{R}^{+}$. For these purposes we use an approach based on diametrical graphs which for the first time was considered in~\cite{PD(UMB)}. The only difference between obtained trees and the canonical trees is that we name vertices of the tree $T_X$ by subsets of the set $X$, which in fact are parts of the corresponding diametrical graphs.

\begin{procedure}\label{p1} If $X$ is a one-point set, then $T_X$ is the rooted tree consisting of the node $X$ with the label~$\operatorname{diam} X=0$. Note that for the rooted trees consisting only of one node, we consider that this node is the root as well as a leaf.

Let $|X|\geqslant 2$. According to Theorem~\ref{t13} we have $G_{D,X} = G_{D,X}[X_1,...,X_k]$ and $k \geqslant 2$. In this case the root of the tree $T_X$ is labeled by $\operatorname{diam} X$ and $T_X$ has the nodes $X_1,...,X_k$ of the first level with the labels
\begin{equation}\label{e2.7}
l(X_i)= \operatorname{diam} X_i,
\end{equation}
$i = 1,\ldots,k$. The nodes of the first level with the label $0$ are leaves, and those indicated by strictly positive labels are internal nodes of the tree $T_X$. If the first level has no internal nodes, then the tree $T_X$ is constructed. Otherwise, by repeating the above-described procedure with each $X_i$ satisfying
$\operatorname{diam} X_i > 0$
instead of $X$, we obtain the nodes of the second level, etc. Since $X$ is finite, all vertices on some level will be leaves, and the construction of $T_X$ is completed.

The above-constructed labeled rooted tree $T_X$ is called the \emph{representing tree} of the ultrametric space $(X, d)$.
\end{procedure}

\begin{definition}\label{d15}
Let $G$ be a nonempty graph, and let $V_0(G)$ be the set (possibly empty) of all isolated vertices of $G$. Denote by $G'$ the subgraph of the graph $G$, induced by set $V(G)\backslash V_0(G)$.
\end{definition}

Let $T$ be a rooted tree with the root $r$. We denote by $\overline{L}_T$ the set of the leaves of $T$, and, for every node $v$ of $T$, by $T_v=T(v)$ the rooted subtree of $T$ with the root $v$ and such that, for every $u\in V(T)$,
\begin{equation}\label{e2.1}
(u \in V(T_v)) \Leftrightarrow (u=v \text{ or $u$ is a successor of $v$ in $T$}).
\end{equation}
If $(X,d)$ is a finite ultrametric space and $T=T_X$ is the representing tree of $X$, then for every node $v\in V(T)$ there are $x_1$, $\ldots$, $x_k \in X$ such that $\overline{L}_{T_v}=\{\{x_1\}, \ldots, \{x_k\}\}$. Thus $\overline{L}_{T_v}$ is a set of one-point subsets of $X$. In what follows we will use the notation $L(T_v)$ for the set $\{x_1, \ldots, x_k\}$.

Let $(X,d)$ be a metric space. Recall that a  \emph{closed ball} with a radius $r \geqslant 0$ and a center $c\in X$ is the set
$$
B_r(c)=\{x\in X\colon d(x,c)\leqslant r\}.
$$
The \emph{ballean} $\mathbf B_X$ of a metric space $(X,d)$ is the set of all closed balls in $(X,d)$. Note that the sets of all open and closed balls in a finite metric space coincide.

The following two statements are fundamental facts from the theory of finite ultrametric spaces.

\begin{proposition}\label{lbpm}
Let $(X,d)$ be a finite nonempty ultrametric space with the representing tree $T_X$. Then the following statements hold.
\begin{itemize}
\item [(i)] $L(T_v)$ belongs to $\mathbf{B}_X$ for every node $v\in V(T_X)$.
\item [(ii)] For every $B \in \mathbf{B}_X$ there exists the unique node $v$ such that $L(T_v)=B$.
\end{itemize}
\end{proposition}
In fact Proposition~\ref{lbpm} claims that the ballean of a finite ultrametric space $(X,d)$ is the vertex set of representing tree $T_X$.

\begin{lemma}\label{l2}
Let $(X, d)$ be a finite ultrametric space with the representing tree $T_X$ whose labeling $l\colon V(T_X)\to \operatorname{Sp}(X)$ is defined by~(\ref{e2.7}) and let $u_1 = \{x_1\}$ and $u_2 = \{x_2\}$ be two different leaves of the tree $T_X$. If
$(u_1, v_1, \ldots, v_n, u_2)$ is the path joining the leaves $u_1$ and $u_2$ in $T_X$, then
\begin{equation}\label{e112}
d(x_1, x_2) = \max\limits_{1\leqslant i \leqslant n} l({v}_i).
\end{equation}
\end{lemma}

The proofs of Proposition~\ref{lbpm} and Lemma~\ref{l2} can be found, for example, in~\cite{P(TIAMM)} and ~\cite{PD(UMB)}, respectively.

\section{Minimum weight spanning tree algorithms}\label{MSPA}

Every finite metric space $(X,d)$ can be considered as a complete weighted graph $(G_X,w)$ if we define
$$
V(G_X)=X \  \text{ and } \  w(\{x,y\})=d(x,y)
$$
for every $x,y \in X$, $x\neq y$.

Recall that a minimum weight spanning tree or simply \emph{minimum spanning tree} $(T_{\min},w)$ of a finite metric space $(X,d)$ is a spanning tree of the graph $(G_X,w)$ such that the sum of weights of all edges of $(T_{\min}, w)$ is smallest among all spanning trees of $(G_X,w)$.

Let us consider the following algorithm described in~\cite{GV12}. Below we adopt it to our terminology.

\textbf{Algorithm 1.} Let $(X,d)$ be a finite ultrametric space and let $(G_X,w)$ be its corresponding weighted graph.

\begin{itemize}
  \item[(i)] Let $G_{D,X}=$ $G_{D,X}[X_1,...,X_k]$ be the diametrical graph of the space $X$. Let us choose in $(G_X,w)$ any $k-1$ edges that would form a spanning tree in the factor-graph obtained from $(G_X,w)$ by contracting each of the $k$ parts $X_1,...,X_k$ to a vertex.
  \item[(ii)] Repeat the same procedure for each of the parts $X_1,...,X_k$ etc., until every
obtained part becomes a vertex.
\end{itemize}

Clearly, every minimum weight spanning tree can be obtained by this procedure.

\begin{remark}[\!\cite{GV12}]\label{r31}
All minimum weight spanning trees of a given $(X,d)$ have one and the same unique weight distribution. Note that all $k-1$ edges chosen at step (i) are of weight $\operatorname{diam} X$ and the edges chosen while considering the parts $X_k$ are of weight $\operatorname{diam} X_k$. Let $\operatorname{Sp}(X)=\{0,d_1,...,d_n\}$. Denote by $I(T_X)$ the set of all inner nodes of the tree $T_X$ and by $S(v)$ the set of all direct successors of the node $v \in V(T_X)$. It is easy to see that for every $i\in \{1,...,n\}$ the corresponding weight $d_i$ appears
\begin{equation}\label{e4}
\sum_{v\in I(T_X) | l(v)=d_i}(|S(v)|-1)
\end{equation}
times in any minimum weight spanning tree.

The uniqueness of weight distribution follows, for example, from Kruskal greedy algorithm~\cite{K56}.
\end{remark}

\begin{remark}[\!\cite{GV12}]\label{r}
For every two points $x,y\in X$ the equality
\begin{equation}\label{e23}
   d(x,y)=\max\{w(e) \, |  \,  e\in E(P_{x,y})\}
\end{equation}
holds, where $P_{x,y}$ is a path connecting the vertices $x$, $y$ in the minimum weight spanning tree $(T_{\min},w)$.
\end{remark}

The following simple algorithm was proposed in~\cite[p. 189]{D82}.

\textbf{Algorithm 2.} Let $(X,d)$ be an ultrametric space with $|X|=n$.  Choose arbitrarily $x_1\in X$. For  $i:=1,..., n-1$ the point $x_{i+1}$ minimizes the distance $d(x,x_{i})$ among the set of points $x\in X$ not belonging to $\{x_1,..,x_{i}\}$.

Below we consider how this algorithm works on representing trees $T_X$ of finite ultrametric spaces $X$. Moreover, we prove that it indeed produces minimum spanning paths.

Recall that the \emph{node level} is the length of the path from the root to the node.

\textbf{Algorithm 3.} Let $(X,d)$ be an ultrametric space with $|X|=n$. Set $i:=1$ and let $x_1$ be an arbitrary leaf of $T_X$. Define $P_1:=\{x_1\}$.
Consider the following procedure for the tree $T_X$.
\begin{itemize}
  \item [(i)] Choose a node $v_i\in V(T_X)$ such that $x_i\in L(T_{v_i})$, $L(T_{v_i})\setminus P_i\neq\varnothing$ and the level of the node $v_i$ is as maximal as possible. Set $i:=i+1$ and go to step (ii).
  \item [(ii)] Choose arbitrary $x_i \in L(T_{v_{i-1}})\setminus P_{i-1}$. Set $P_{i}:=P_{i-1}\cup\{x_i\}$. If $i=n$, then construction of the path $P$ is completed else go to step (i).
\end{itemize}

\begin{remark}\label{r23} Note that condition (i) means that we choose the smallest ball of the space $X$ that contains the vertex $x_i$ and contains at least one vertex which does not belong to $P_i$. According to the construction of $T_X$ and Lemma~\ref{l2} the distance $d\{x_i,x_{i+1}\}$  is equal to the label $l(v_i)$, $i=1,\ldots,n-1$.
Using Lemma~\ref{l2} and the fact that labels of vertices of representing trees strictly decrease on the paths from the root to any leaves we see that the minimum of the value $d(x,x_i)$  $x \not \in \{x_1,...,x_{i}\}$ is achieved at any $x\in L(T_{v_{i}})\setminus P_{i}$ (compare with Algorithm 2).

It is easy to see that one and the same node can be chosen several times, i.e., the equality $v_i=v_j$, $i\neq j$ is possible. By step (ii) all elements of the set $P_n$ are different, i.e., $P_n=L_{T_X}=X$. The existence of a vertex $v_i$ such that the conditions $x_i\in L(T_{v_i})$, $L(T_{v_i})\setminus P_i\neq\varnothing$ follows from the existence of the root of the tree $T_X$.
\end{remark}

\begin{example}
The rooted representing tree $T_Z$ of an ultrametric space $Z$ and the corresponding minimum spanning path $P$ are depicted in Figures~\ref{fig2} and~\ref{fig3}, respectively. To demonstrate the work of Algorithm 3 it is sufficient to indicate the sequence of inner nodes chosen at step (i). Since all the inner nodes of $T_Z$ are marked by different labels we can represent the sequence in the following form: $l(v_1)=1$; $l(v_2)=1$; $l(v_3)=4$; $l(v_4)=4$; $l(v_5)=2$; $l(v_6)=9$; $l(v_7)=3$; $l(v_8)=5$; $l(v_9)=9$;
$l(v_{10})=7$; $l(v_{11})=7$; $l(v_{12})=8$; $l(v_{13})=8$;  $l(v_{14})=6$.
\end{example}

\begin{figure}[htb]
\begin{tikzpicture}[scale=1.4]
\tikzstyle{level 1}=[level distance=10mm,sibling distance=2.4cm]
\tikzstyle{level 2}=[level distance=10mm,sibling distance=0.8cm]
\tikzstyle{level 3}=[level distance=10mm,sibling distance=.6cm]
\tikzset{small node/.style={circle,draw,inner sep=0.7pt,fill=black}}
\tikzset{solid node/.style={circle,draw,inner sep=1.5pt,fill=black}}
\tikzset{white node/.style={circle,draw,inner sep=1.5pt,fill=white}}
\tikzset{common node/.style={circle,draw,inner sep=1.5pt,fill=gray}}

\node at (0,0) [small node, label=above: {\tiny $v_i$}] {}
child{node [small node, label=left: {\tiny $v_{i_1}$}] {}
	child{node [small node, label=right: {\tiny $$}] {}
    }
   	child{node [small node, label=below: {\scriptsize $$}] {}}
	child{node [small node, label=right: {\tiny $x_i$}] {}
    }
}
child{node [small node, label=right: {\tiny $v_{i_2}$}] {}
	child{node [small node, label=right: {\tiny $x_{i+1}$}] {}
    }
	child{node [small node, label=below: {\scriptsize $$}] {}}
}
child{node [small node, label=right: {\tiny $v_{i_k}$}] {}
	child{node [small node, label=right: {\tiny $$}] {}
    }
   	child{node [small node, label=below: {\scriptsize $$}] {}}
	child{node [small node, label=right: {\scriptsize $x_m$}] {}
    }
};
\end{tikzpicture}
\caption{The subtree $T_{v_i}$ of the tree $T_X$.}
\label{fig0}
\end{figure}
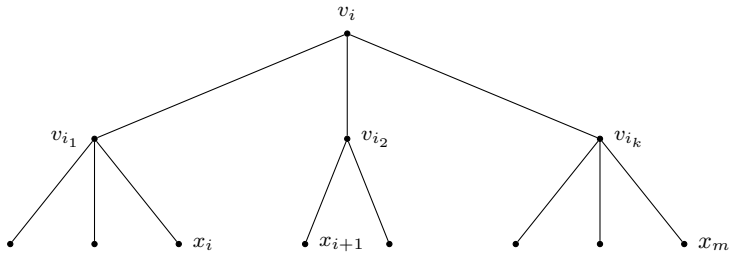

\begin{proposition}\label{p25}
Every path produced by Algorithm 3 is a minimum spanning path.
\end{proposition}
\begin{proof}
According to Remark~\ref{r31} to show that the path $P_n$ is a minimum weight spanning path it suffices to show that $P_n$ has the weight distribution described by~(\ref{e4}). Indeed, let the node $v_i$ was chosen the first time at step (i), see Figure~\ref{fig0}. This means that some $x_i\in L(T_{v_i})$ was just chosen. Let $|S(v_i)|=k$ and let $v_{i_1},...,v_{i_k}$ be direct successors of $v_i$. Without loss of generality consider that $x_i \in L(T_{v_{i_1}})$. According to step (i)
all elements of the ball $L(T_{v_{i_1}})$ are already added to the path $P_i$.
Otherwise, we must chose $v_{i_{1}}$ instead of $v_i$.
Further we must chose the next point $x_{i+1}$ of the path belonging to $L(T_{v_i})\setminus L(T_{v_{i_1}})$. Without loss of generality consider that $x_{i+1}\in L(T_{v_{i_2}})$, etc. Let $x_m$ be the last one point chosen from the ball $L(T_{v_{i_k}})$.
According to Lemma~\ref{l2} in the part $(x_1,...,x_i,x_{i+1},...,x_m)$ of the path $P_n$ the weight $l(v_i)=d(x_i, x_{i+1})$ appears $|S(v_i)|-1=k-1$ times only during crossing from one of the balls $L(T_{v_{i_1}}),...,L(T_{v_{i_k}})$ to another. Taking into consideration that another vertices with the label $l(v_i)$ may exist in the tree $T_X$ we see that distribution~(\ref{e4}) holds for $d_i=l(v_i)$.
\end{proof}

Let $(X,d)$ be an ultrametric space with $|X|=n$. Denote by $\mathfrak{P}_{min}(X)$ the set of all  minimum spanning paths of the graph $(G_X,w)$. Let $P\in \mathfrak{P}_{min}(X)$, $V(P)=\{x_1,....,x_n\}$. Denote by \emph{path spectrum} $\mathcal{PS}(P)$  the following finite sequence of numbers $(s_1,...,s_{n-1})$ where $s_i=w(\{x_i, x_{i+1}\})=d(x_i, x_{i+1})$, $i=1,...,n-1$.

\begin{remark}\label{r55}
Using the representing tree $T_X$ of a finite ultrametric space $X$ it is easy to see which subsets of $X$ are balls in the space $X$. The situation with minimum spanning paths is also enough clear. Let $P\in \mathfrak{P}_{min}(X)$ with $V(P)=\{x_1,...,x_n\}$ and $\mathcal{PS}(P)=(s_1,..,s_{n-1})$. Using Remark~\ref{r} one easily can establish that there is one-to-one correspondence between nonsingular balls of the space $X$ and subsets of consecutive points $x_i, x_{i+1},..., x_{i+k}$, $i,k\geqslant 1$, $i+k\leqslant n$ that satisfy the following property
$$
s_{l}<s_{i-1}\, \text{ and } s_l<s_{i+k} \, \text{ for every } \, l\in \{i,...,i+k-1\}.
$$
Naturally, in the case $i=1$ only the inequality $s_l<s_{i+k}$ must hold as well as in the case $i+k=n$ only the inequality $s_l<s_{i-1}$.
 In the extremal case $i=1$ and $i+k=n$ we obtain that $\{x_1,...,x_n\}=X \in \mathbf B_X$.
 In other words, the subsequence $(x_i,x_{i+1}, ... ,x_{i+k})$ forms a ball if and only if all the weights inside this part of the path $P$ are surrounded by bigger weights $s_{i-1}$ and $s_{i+k}$. Clearly, the diameter of the ball $\{x_i, x_{i+1},..., x_{i+k}\}$ is the maximal value among $s_i,...,s_{i+k-1}$.
\end{remark}

\section{Minimum spanning paths for special classes of finite ultrametric spaces}\label{MSTFSC}
It is possible to distinguish classes of finite ultrametric spaces applying some restrictions on their representing trees.
Finite ultrametric spaces $X$ for which their representing trees $T_X$ are strictly binary,
have injective internal labeling, are extremal for the Gomory-Hu inequality, are as rigid as possible were considered in~\cite{DPT}, \cite{DP20}, \cite{PD(UMB)} and~\cite{DPT(Howrigid)}, respectively.
The aim of this section is to give a characterization of the above mentioned classes in terms of special type minimum spanning paths.

\subsection{ Spaces for which $T_X$ is strictly binary.}
Recall that a rooted tree is strictly binary if its every internal node has exactly two children.

\begin{proposition}[\!\cite{DPT}]\label{p10}
Let $(X,d)$ be a finite nonempty ultrametric space. The following conditions are
equivalent.
\begin{itemize}
\item[(i)] $T_X$ is strictly binary.
\item[(ii)] If $Y\subseteq X$ and $|Y|\geqslant 3$, then there exists a Hamilton cycle $C \subseteq (G_{Y},w)$ with exactly two edges of maximal weight.
\item[(iii)] There are no equilateral triangles in $(X,d)$.
\end{itemize}
\end{proposition}

\begin{proposition}\label{p46}
Let $(X,d)$ be a finite nonempty ultrametric space. Then  for any $P\in \mathfrak{P}_{min}(X)$ with $V(P)=\{x_1,..,x_n\}$ and
$\mathcal{PS}(P)=(s_1,...,s_{n-1})$ the following conditions are equivalent.
\begin{itemize}
\item[(i)] $T_X$ is strictly binary.
\item[(ii)] If $s_i=s_j$, $1\leqslant i<j\leqslant n-1$, then there exists $k$, $i<k<j$, such that $s_k>s_i=s_j$.
\end{itemize}
\end{proposition}

\begin{proof}
(i)$\Rightarrow$(ii) Let $(X,d)$ be a finite ultrametric space and let $T_X$ be strictly binary. Suppose that (ii) does not hold for some $P\in \mathfrak{P}_{min}(X)$. Consequently, if there exist $i,j$, $i<j$ such that  $s_i=s_j$, then for all $k$ such that $i<k<j$ the inequality $s_k<s_i=s_j$ holds. Suppose first that $i<j-1$ and such $k$ do exists. Consider a triplet of points $(x_{i}, x_j, x_{j+1})$. By~(\ref{e23}) we have the equalities $d(x_{i}, x_j)=d(x_j, x_{j+1})=d(x_{i}, x_{j+1})=s_i$ which contradicts to condition (iii) of Proposition~\ref{p10}. Let now $i=j-1$. Then for the triplet $(x_i,x_j,x_{j+1})$ we have  the same contradiction.

(ii)$\Rightarrow$(i) Conversely, let condition (ii) hold and let $T_X$ be not strictly binary. Consequently, by condition (iii) of Proposition~\ref{p10} for every $(P,w)\in \mathfrak{P}_{min}(X)$ there exist three points $x_i, x_j, x_k \in V(P)$, $i<j<k$ such that $d(x_i,x_j)=d(x_j,x_k)=d(x_i,x_k)$. By~(\ref{e23}) there exists an edge $e_1\in E(P)$ of maximal weight among the edges of the subpath of the path $P$ with vertices $\{x_i,...,x_j\}$. Analogously, there exists an edge $e_2\in E(P)$ with the same property belonging to the subpath with vertices $\{x_j,...,x_k\}$. Obviously, there are only two possibilities: 1) $e_1$ and $e_2$ are adjacent; 2) the weight of every edge which is between $e_1$ and $e_2$ is less or equal to $w(e_1)=w(e_2)$, which contradicts to condition (ii).
\end{proof}

\subsection{Spaces for which $T_X$ has injective internal labeling.}

We shall say that internal labeling of a representing tree $T_X$ is injective if the labels of different internal nodes of $T_X$ are different.

\begin{theorem}[\!~\cite{DP20}]\label{t1}
Let $(X,d)$ be a finite nonempty ultrametric space. The following conditions are equivalent.
\begin{itemize}
\item [(i)] The diameters of different nonsingular balls are different.
\item [(ii)] The internal labeling of $T_X$ is injective.
\item [(iii)] $G'_{t,X}$ is a complete multipartite graph for every $t\in \operatorname{Sp}(X)\setminus \{0\}$.
\item [(iv)] The equality
$$
|\operatorname{Sp}(X)| = |\mathbf{B}_X| - |X| + 1
$$
holds.
\end{itemize}
\end{theorem}

\begin{proposition}\label{p48}
Let $(X,d)$ be a finite nonempty ultrametric space. Then  for any $P\in \mathfrak{P}_{min}(X)$ with $V(P)=\{x_1,..,x_n\}$ and
$\mathcal{PS}(P)=(s_1,...,s_{n-1})$ the following conditions are equivalent.
\begin{itemize}
\item[(i)] The labels of different internal nodes of $T_X$ are different.
\item[(ii)] If $s_i=s_j$, $1\leqslant i < j \leqslant n-1$, then the inequality $s_k\leqslant s_i=s_j$ holds for every $k$ such that $i<k<j$.
\end{itemize}
\end{proposition}

\begin{proof}
(i)$\Rightarrow$(ii) Let $(X,d)$ be a finite ultrametric space and let all the labels of different internal nodes of $T_X$ are different. Suppose that (ii) does not hold for some $P\in \mathfrak{P}_{min}(X)$. Consequently, if there exist $i,j$, $i<j$ such that  $s_i=s_j$, then there exists at least one $k$ such that $i<k<j$ and the inequality $s_k>s_i=s_j$ holds. Let $k_1$ be the smallest integer such that $i<k_1$ and $s_i<s_{k_1}$. Analogously, let $k_2$ be the largest integer such that $k_2<j$ and $s_j<s_{k_2}$. (Possibly $k_1=k_2$.)
Let also $l_1$ be the largest integer such that $l_1<i$ and $s_{l_1}>s_{i}$. Analogously, let $l_2$ be the smallest integer such that $j<l_2$ and $s_j<s_{l_2}$. By Remark~\ref{r55} the sets
$\{x_{l_1+1},...,x_i,...,x_{k_1}\}$ and $\{x_{k_2+1},...,x_j,...,x_{l_2}\}$ are different nonsingular balls in $X$ with the equal diameters $s_i=s_j$ which contradicts to condition (i) of Theorem~\ref{t1}.
Note also that if there is no $l_1$ ($l_2$) with the above described properties, then the set $\{x_{1},...,x_i,...,x_{k_1}\}$ ($\{x_{k_2+1},...,x_j,...,x_{n}\}$) will form the respective desired ball.

(ii)$\Rightarrow$(i) Let condition (ii) hold. Suppose that there are two different internal nodes $v_1$ and $v_2$ with $l(v_1)=l(v_2)$. Since labels on representing trees strictly decrease from the root to any vertex we have that $v_1$ is not a successor of $v_2$ and $v_2$ is not a successor of $v_1$.  Let $P\in \mathfrak{P}_{min}(X)$ with $V(P)=\{x_1,...,x_n\}$ and $\mathcal{PS}(P)=(s_1,...,s_{n-1})$ be any path constructed using Algorithm 3. Hence there exist $i,j$, $i<j-1$, such that $d(x_i, x_{i+1})=d(x_j, x_{j+1})$, (i.e., $s_i=s_j$) where $x_i,x_{i+1}$ and $x_j, x_{j+1}$ are successors of $v_1$ and $v_2$, respectively. By Lemma~\ref{l2} we have $d(x_{i+1}, x_{j})>l(v_1)=l(v_2)$. According to Remark~\ref{r} there exists $k$, $i<k<j$, such that $s_k>s_i=s_j$ which contradicts to condition (ii).
\end{proof}

\subsection{Spaces extremal for the Gomory-Hu inequality.}
In 1961 E.\,C.~Gomory and T.\,C.~Hu \cite{GomoryHu(1961)} for arbitrary finite ultrametric space $X$ proved the inequality \mbox{$|\operatorname{Sp}(X)| \leqslant |X|$}. Define by $\mathfrak{U}$ the class of finite ultrametric spaces $X$ with $|\operatorname{Sp}(X)| = |X|$.
There exist descriptions of $X \in \mathfrak{U}$ in terms of graphs $G'_{r,X}$~\cite{PD(UMB), DP18}, in terms of representing trees~\cite{PD(UMB)} and in terms of weighted Hamiltonian cycles and weighted Hamiltonian paths~\cite{DPT}.

Below we need the following criterion.
\begin{theorem}[\!\cite{PD(UMB)}]\label{t22}
Let $(X, d)$ be a finite ultrametric space with $|X| \geqslant 2$. The following conditions are equivalent.
\begin{itemize}
  \item [(i)] $(X, d) \in \mathfrak U$.
  \item [(iii)] $T_X$ is strictly binary and the labels of different internal nodes are different.
\end{itemize}
\end{theorem}

\begin{proposition}\label{p44}
Let $(X,d)$ be a finite nonempty ultrametric space. The following conditions are
equivalent.
\begin{itemize}
\item[(i)] $X \in \mathfrak{U}$.
\item[(ii)] All elements of $\mathcal{PS}(P)$ are different for every $P\in \mathfrak{P}_{min}(X)$.
\end{itemize}
\end{proposition}

\begin{proof}
According to Theorem~\ref{t22} condition (i) of this proposition is equivalent to the intersection of conditions (i) of Proposition~\ref{p46} and (i) of Proposition~\ref{p48}. This in turn is equivalent to the intersection of conditions (ii) of Proposition~\ref{p46} and (ii) of Proposition~\ref{p48}. The last two conditions give condition (ii) of this proposition. This completes the proof.
\end{proof}

\subsection{Ultrametric spaces which are as rigid as possible.}

Let $(X, d)$ be a metric space and let $\operatorname{Iso}(X)$ be the group of isometries of $(X, d)$. For every self-map $f\colon X\to X$ we denote by $\operatorname{Fix}(f)$ the set of fixed points of $f$.
The finite ultrametric spaces satisfying the equality
\begin{equation*}
\min_{g \in \operatorname{Iso}(X)} |\operatorname{Fix}(g)| = |X|-2,
\end{equation*}
are called as rigid as possible, see~\cite{DPT(Howrigid)} for detailed definitions. Denote by $\mathfrak{R}$ this class of spaces. The following theorem gives us a characterization of spaces from the class $\mathfrak{R}$ in terms of representing trees.

\begin{theorem}[\!~\cite{DPT(Howrigid)}]\label{t38}
Let $(X, d)$ be a finite ultrametric space with $|X| \geq 2$. Then the following
statements are equivalent.
\begin{itemize}
\item [(i)] $(X, d) \in \mathfrak{R}$.
\item [(ii)] $T_X$ is strictly binary with exactly one inner node at each level except the last level.
\end{itemize}
\end{theorem}

\begin{proposition}\label{p410x}
Let $(X,d)$ be a finite nonempty ultrametric space. The following conditions are equivalent.
\begin{itemize}
\item[(i)] $(X, d) \in \mathfrak{R}$.
\item[(ii)] There exists a path $P\in \mathfrak{P}_{min}(X)$ such that the sequence $\mathcal{PS}(P)=(s_1,...,s_{n-1})$ is strictly monotone.
\end{itemize}
\end{proposition}

\begin{proof}
The proof easily follows from the structure of representing tree $T_X$ described by condition (ii) of Theorem~\ref{t38}.
\end{proof}

\begin{proposition}\label{p41}
Let $X$ be an ultrametric space. Any minimum spanning tree $T_{\min}$ of the space $X$ is a path if and only if the representing tree $T_X$ is isomorphic as rooted tree to one of the rooted trees $T_1,..., T_5$ depicted in Figure~\ref{fig1}.
\end{proposition}

\begin{proof}
 For the proof of this proposition refer to Algorithm 1. Let for the space $(X,d)$ the diametrical graph $G_{D,X}[X_1, X_2]$ be bipartite. It is easy to see that in order to avoid a construction of a minimum spanning tree which is not a path every part $X_1$ and $X_2$ must contain no more than two points. This observation describes cases $T_1,...,T_3$. The case when $G_{D,X}$ has four parts and more is not admissible since there exists a four-point tree which is not a path. The case $T_5$ is trivial and the case $T_4$ is left to the reader.
\end{proof}

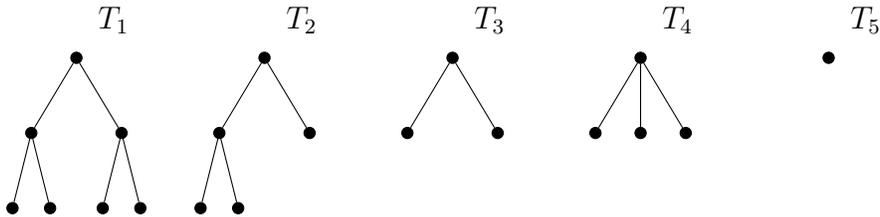
\begin{figure}[htb]
\begin{tikzpicture}[scale=1]
\tikzstyle{level 1}=[level distance=10mm,sibling distance=1.2cm]
\tikzstyle{level 2}=[level distance=10mm,sibling distance=0.5cm]
\tikzstyle{level 3}=[level distance=10mm,sibling distance=.5cm]
\tikzset{solid node/.style={circle,draw,inner sep=1.5pt,fill=black}}

\node at (0,0.5) [label=right:\(T_1\)] {};
\node at (0,0) [solid node] {}
child{node [solid node] {}
	child{node [solid node] {}}
	child{node [solid node] {}}
}
child{node [solid node] {}
	child{node [solid node] {}}
	child{node [solid node] {}}
};

\node at (2.5,0.5) [label=right:\(T_2\)] {};
\node at (2.5,0) [solid node] {}
child{node [solid node] {}
	child{node [solid node] {}}
	child{node [solid node] {}}
}
child{node [solid node] {}
};

\node at (5,0.5) [label=right:\(T_3\)] {};
\node at (5,0) [solid node] {}
child{node [solid node] {}
}
child{node [solid node] {}
};

\tikzstyle{level 1}=[level distance=10mm,sibling distance=0.6cm]
\node at (7.5,0.5) [label=right:\(T_4\)] {};
\node at (7.5,0) [solid node] {}
child{node [solid node] {}
}
child{node [solid node] {}
}
child{node [solid node] {}
};

\node at (10,0.5) [label=right:\(T_5\)] {};
\node at (10,0) [solid node] {};

\end{tikzpicture}
\caption{Rooted representing trees.}
\label{fig1}
\end{figure}

\textbf{Open problem.}
By Remark~\ref{r}, every finite ultrametric space $(X,d)$ is defined by any minimum spanning path  $P\in \mathfrak{P}_{min}(X)$. In other words any finite ultrametric space $X$ with $|X|=n$ is defined by the sequence $(s_1,...,s_{n-1})$, $s_i\in \operatorname{Sp}(X)\setminus \{0\}$. Moreover, such sequence is not unique. It is clear that it is not possible to reconstruct the whole space $X$ having only its spectrum $\operatorname{Sp}(X)$. This indeterminacy hints to the following question: what can we say about the space $X$ or its representing tree $T_X$ knowing only its spectrum $\operatorname{Sp}(X)$?
For example, by Theorem~\ref{t22} the extremal case $|\operatorname{Sp}(X)|=|X|$ allows us to establish some properties of $T_X$.
More strict variation: what can we say about the space $X$ knowing only its multispectrum? Under \emph{multispectrum} of $(X,d)$ we understand the following set $\{(d_i, k_i) \,  |  \, d_i \in \operatorname{Sp}(X)\}$, where $k_i$ is the number of times when $d_i$ appears among values of the ultrametric $d$. Such formulation of the problem is similar to the one of the tasks of spectral graph theory: study of the  properties of a graph in relationship to the eigenvalues of matrices associated with the graph.

\section{Hausdorff distance and minimum spanning paths}\label{HDandMST}
The main result of this section is Theorem~\ref{t44} which gives us an explicit formula for Hausdorff distance between two subsets of a finite ultrametric space. The idea is the following.  First, given two subsets $X$ and $Y$ of the finite ultrametric space $Z$ we define a set of balls $\mathfrak{B}_{XY} \subseteq \mathbf{B}_Z$ depending on these subsets, see Definition~\ref{d21}. Then in Theorem~\ref{t44} we show that the Hausdorff distance  $d_H(X,Y)$ is equal to the maximal diameter of balls from $\mathfrak{B}_{XY}$.

Let us recall the necessary definitions. Let $(Z,d)$ be a finite  ultrametric space. For any two subsets $X$ and $Y$ of $Z$
the distance between $X$ and $Y$ is
$$
\operatorname{dist}(X, Y) = \min \{d(x, y)\colon x \in X , y \in Y\},
$$
in particular, for $x\in Z$, $\operatorname{dist}(x, X) = \operatorname{dist} (\{x\}, X)$. For a set $A \subset Z$ and $\varepsilon >0$, its $\varepsilon$-neighborhood is the set
$$ U_{\varepsilon} (A) = \{ x \in X\colon
\operatorname{dist} (x, A) < \varepsilon \}.$$

The Hausdorff distance between $X$ and $Y$ is defined by
\begin{equation}\label{e41}
d_{H} (X, Y) = \inf \{\varepsilon > 0 \colon X \subset
U_{\varepsilon} (Y) \ \text{ and } \ Y\subset U_{\varepsilon} (X) \}.
\end{equation}

Recall that in ultrametric space every ball is a union of disjoint balls.
For finite ultrametric spaces $X$ this fact easily follows from the construction of representing tree $T_X$ and Proposition~1.7. In particular, Proposition~1.7 and Lemma~1.8 imply that for every nonsingular ball $B \in \mathbf{B}_X$ the diametrical graph $G_{D,B}=G_{D,B}[B_1,...,B_n]$ is a complete $n$-partite graph with the parts $B_1,...,B_n$, where all the subsets $B_1,...,B_n$ are balls of the space $X$. In this case, clearly, $B=B_1\cup...\cup B_n$.

\begin{definition}\label{d21}
Let $(Z,d)$ be an ultrametric spaces and let $X,Y$, $X\neq Y$, be some subsets of $Z$.
Denote by $\mathfrak B_{XY}\subseteq \mathbf B_Z$ the set of all balls $B$ in the space $Z$ having simultaneously the following two properties.
\begin{itemize}
  \item [(i)] ($(X\setminus Y)\cap B\neq \varnothing \neq Y \cap B$) or ($(Y\setminus X)\cap B\neq \varnothing \neq X \cap B$).
  \item [(ii)] 
  Let
  \begin{equation}\label{e1}
  G_{D,B} = G_{D,B}[B_1,...,B_n].
  \end{equation}
  be the diametrical graph of the subspace $B\subseteq Z$.
  Then there exists at least one ball $B_k$, $k\in \{1,...,n\}$, such that
  $$
  (B_k\cap (X\setminus Y)\neq \varnothing \text{ and } B_k\cap Y=\varnothing)
\text{ \emph{xor} }
  (B_k\cap (Y\setminus X)\neq \varnothing \text{ and } B_k\cap X=\varnothing).
$$
\end{itemize}
As usual, $xor$ here is exclusive disjunction.
\end{definition}

\begin{remark}
It follows from condition (i) that if $|B|=1$, then $B\notin \mathfrak B_{XY}$.
\end{remark}

As usual  $X \Delta Y= (X\cup Y)\setminus(X \cap Y)$ is the symmetric difference of the sets $X$ and $Y$

\begin{lemma}\label{l1}
Let $Z$ be a finite ultrametric space and let $X, Y$, $X\neq Y$, be some nonempty subsets of $Z$. Then for every $x\in X\Delta Y$ there exists a ball $B\in \mathfrak B_{XY}$ such that $x\in B$.
\end{lemma}

\begin{proof}
Without loss of generality, consider that $x\in X\setminus Y$. Let $\{x\}=\tilde{B}_1\subset \tilde{B}_2\subset \dots \tilde{B}_{n-1}\subset \tilde{B}_n =Z$ be a sequence of all different balls containing $x$ and let $\tilde{B}_i \notin \mathfrak{B}_{XY}$ for all $i\in 2,...,n$. Clearly, condition (i) of Definition~\ref{d21} holds for $B=\tilde{B}_n$. Hence, condition (ii) of the same definition does not hold for $B=\tilde{B}_n$. Consequently,  for every ball $B_i$, $i\in \{1,...,n\}$ from~(\ref{e1}) only one of three following possibilities holds:

1) $B_i\cap(X\cup Y)=\varnothing$,

2) $X\cap Y \cap B_i\neq \varnothing$ and $(X\Delta Y)\cap B_i=\varnothing$,

3) Condition (i) of Definition~\ref{d21} holds with $B=B_i$.

Since $x\in \tilde{B}_{n-1}$, $x\in X\setminus Y$, and $\tilde{B}_{n-1}=B_i$ for some $i\in \{1,...,n\}$ we have that condition (i) holds for $B_i=\tilde{B}_{n-1}$. Repeating these considerations with every $\tilde{B}_{i}$ we see that even if all $\tilde{B}_{i}\notin \mathfrak{B}_{XY}$, $i=3,..,n$, the ball $\tilde{B}_{2}$ must belong to $\mathfrak{B}_{XY}$ since condition (i) holds for this ball by supposition of this procedure and condition (ii) holds because we can take $B_k=\{x\}$ in (ii).
\end{proof}

\begin{theorem}\label{t44}
Let $(Z,d)$ be an ultrametric space. Then for nonempty subsets $X,Y\subseteq Z$ the Hausdorff distance can be calculated as follows:
\begin{equation}\label{e45}
d_H(X,Y)=
\begin{cases}
0, &X=Y,\\
\max\limits_{B\in \mathfrak{B}_{XY}}\operatorname{diam} B, &\text{otherwise}.
\end{cases}
\end{equation}
\end{theorem}

\begin{proof}
If $X=Y$, then the equality $d_H(X,Y)=0$ is evident. Suppose $X\neq Y$.
Let us show that for $r= \max\limits_{B\in \mathfrak{B}_{XY}}\operatorname{diam} B$ the relations $X\subset U_{r}(Y)$ and $Y\subset U_{r}(X)$ hold.
Let $y\in Y$. If $y\in X\cap Y$, then, clearly, $y\in U_r(X)=\bigcup\limits_{x\in X}B_r(x)$ for any $r>0$. Let $y\in Y\setminus X$. Consequently, Lemma~\ref{l1} implies that $y$ belongs to some ball $B\in \mathfrak{B}_{XY}$. According to condition (i) of Definition~\ref{d21} there exists  $x\in X$ such that $x\in B$. Since $\operatorname{diam} B\leqslant r$ we have $y\in B_r(x)$. Consequently, $y\in U_r(X)$. The inclusion $X\subset U_{r}(Y)$ can be shown analogously.

Let now $r<\max\limits_{B\in \mathfrak{B}_{XY}}\operatorname{diam} B$. Consequently, there exists $B\in \mathfrak{B}_{XY}$ such that $\operatorname{diam} B>r$. Let us show that at least one of the relations  $X\subset U_{r}(Y)$ or $Y\subset U_{r}(X)$ does not hold. Without loss of generality, in virtue of condition (ii) of Definition~\ref{d21} consider that in decomposition~(\ref{e1}) for the ball $B$ there exists a ball $B_k$ such that $B_k\cap (X\setminus Y)\neq \varnothing$ and
\begin{equation}\label{e3}
  B_k\cap Y=\varnothing.
\end{equation}
Let $x\in B_k\cap (X\setminus Y)$. Using Lemma~\ref{l2} we see that by~(\ref{e1}) and by~(\ref{e3}), the equality $d(x,y) = \operatorname{diam} B > r$ holds for all $y\in B\cap Y$. Moreover, using the property that the labels of a representing tree strictly decrease on any path from the root to a leaf, we see that the inequality $d(x,y) > \operatorname{diam} B$ holds for all $y \in Y\setminus B$. Hence $x$ does not belong to $U_r(Y)$.
Thus, the infinum in~(\ref{e41}) is achieved when $\varepsilon = \max\limits_{B\in \mathfrak{B}_{XY}}\operatorname{diam} B$.
\end{proof}

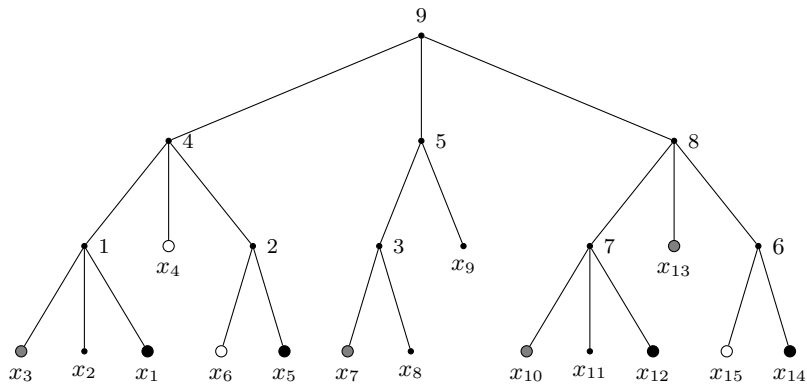
\begin{figure}[htb]
\begin{tikzpicture}[scale=1.4]
\tikzstyle{level 1}=[level distance=10mm,sibling distance=2.4cm]
\tikzstyle{level 2}=[level distance=10mm,sibling distance=0.8cm]
\tikzstyle{level 3}=[level distance=10mm,sibling distance=.6cm]
\tikzset{small node/.style={circle,draw,inner sep=0.7pt,fill=black}}
\tikzset{solid node/.style={circle,draw,inner sep=1.5pt,fill=black}}
\tikzset{white node/.style={circle,draw,inner sep=1.5pt,fill=white}}
\tikzset{common node/.style={circle,draw,inner sep=1.5pt,fill=gray}}

\node at (0,0) [small node, label=above: {\tiny $9$}] {}
child{node [small node, label=right: {\tiny $4$}] {}
	child{node [small node, label=right: {\tiny $1$}] {}
        	child{node [common node, label=below: {\scriptsize $x_3$}]{}}
            child{node [small node, label=below: {\scriptsize $x_2$}] {}}
            child{node [solid node, label=below: {\scriptsize $x_1$}] {}}
    }
   	child{node [white node, label=below: {\scriptsize $x_4$}] {}}
	child{node [small node, label=right: {\tiny $2$}] {}
        	child{node [white node, label=below: {\scriptsize $x_6$}] {}}
            child{node [solid node, label=below: {\scriptsize $x_5$}] {}}
    }
}
child{node [small node, label=right: {\tiny $5$}] {}
	child{node [small node, label=right: {\tiny $3$}] {}
        	child{node [common node, label=below: {\scriptsize $x_7$}] {}}
        	child{node [small node, label=below: {\scriptsize $x_8$}] {}}
    }
	child{node [small node, label=below: {\scriptsize $x_9$}] {}}
}
child{node [small node, label=right: {\tiny $8$}] {}
	child{node [small node, label=right: {\tiny $7$}] {}
        	child{node [common node, label=below: {\scriptsize $x_{10}$}] {}}
            child{node [small node, label=below: {\scriptsize $x_{11}$}] {}}
            child{node [solid node, label=below: {\scriptsize $x_{12}$}] {}}
    }
   	child{node [common node, label=below: {\scriptsize $x_{13}$}] {}}
	child{node [small node, label=right: {\scriptsize $6$}] {}
        	child{node [white node, label=below: {\scriptsize $x_{15}$}] {}}
            child{node [solid node, label=below: {\scriptsize $x_{14}$}] {}}
    }
};
\end{tikzpicture}
\caption{Rooted representing tree $T_Z$.}
\label{fig2}
\end{figure}

\begin{figure}[htb]
\begin{center}
\begin{tikzpicture}
\tikzset{small node/.style={circle,draw,inner sep=0.7pt,fill=black}}
\tikzset{solid node/.style={circle,draw,inner sep=1.5pt,fill=black}}
\tikzset{white node/.style={circle,draw,inner sep=1.5pt,fill=white}}
\tikzset{common node/.style={circle,draw,inner sep=1.5pt,fill=gray}}

\def\xx{0cm};
\def\yy{0cm};
\def\dx{0.8cm};

\foreach \i in {1}{
\draw (\i*\dx, \yy) node (\i) [below] {$x_{\i}$} -- (\i*\dx+\dx,\yy);
\draw [fill = black] (\i*\dx, \yy) circle [radius=2pt];
}

\foreach \i in {2}{
\draw (\i*\dx, \yy) node (\i) [below] {$x_{\i}$} -- (\i*\dx+\dx,\yy);
\draw [fill = black] (\i*\dx, \yy) circle [radius=1pt];
}

\foreach \i in {3}{
\draw (\i*\dx, \yy) node (\i) [below] {$x_{\i}$} -- (\i*\dx+\dx,\yy);
\draw [fill = gray] (\i*\dx, \yy) circle [radius=2pt];
}

\foreach \i in {4}{
\draw (\i*\dx, \yy) node (\i) [below] {$x_{\i}$} -- (\i*\dx+\dx,\yy);
\draw [fill = white] (\i*\dx, \yy) circle [radius=2pt];
}

\foreach \i in {5}{
\draw (\i*\dx, \yy) node (\i) [below] {$x_{\i}$} -- (\i*\dx+\dx,\yy);
\draw [fill = black] (\i*\dx, \yy) circle [radius=2pt];
}

\foreach \i in {6}{
\draw (\i*\dx, \yy) node (\i) [below] {$x_{\i}$} -- (\i*\dx+\dx,\yy);
\draw [fill = white] (\i*\dx, \yy) circle [radius=2pt];
}

\foreach \i in {7}{
\draw (\i*\dx, \yy) node (\i) [below] {$x_{\i}$} -- (\i*\dx+\dx,\yy);
\draw [fill = gray] (\i*\dx, \yy) circle [radius=2pt];
}

\foreach \i in {8}{
\draw (\i*\dx, \yy) node (\i) [below] {$x_{\i}$} -- (\i*\dx+\dx,\yy);
\draw [fill = black] (\i*\dx, \yy) circle [radius=1pt];
}

\foreach \i in {9}{
\draw (\i*\dx, \yy) node (\i) [below] {$x_{\i}$} -- (\i*\dx+\dx,\yy);
\draw [fill = black] (\i*\dx, \yy) circle [radius=1pt];
}

\foreach \i in {10}{
\draw (\i*\dx, \yy) node (\i) [below] {$x_{\i}$} -- (\i*\dx+\dx,\yy);
\draw [fill = gray] (\i*\dx, \yy) circle [radius=2pt];
}

\foreach \i in {11}{
\draw (\i*\dx, \yy) node (\i) [below] {$x_{\i}$} -- (\i*\dx+\dx,\yy);
\draw [fill = black] (\i*\dx, \yy) circle [radius=1pt];
}

\foreach \i in {12}{
\draw (\i*\dx, \yy) node (\i) [below] {$x_{\i}$} -- (\i*\dx+\dx,\yy);
\draw [fill = black] (\i*\dx, \yy) circle [radius=2pt];
}

\foreach \i in {13}{
\draw (\i*\dx, \yy) node (\i) [below] {$x_{\i}$} -- (\i*\dx+\dx,\yy);
\draw [fill = gray] (\i*\dx, \yy) circle [radius=2pt];
}

\foreach \i in {14}{
\draw (\i*\dx, \yy) node (\i) [below] {$x_{\i}$} -- (\i*\dx+\dx,\yy);
\draw [fill = black] (\i*\dx, \yy) circle [radius=2pt];
}

\foreach \i in {15}{
    \draw (\i*\dx, \yy) node (\i) [below] {$x_{\i}$};
    \draw [fill = white] (\i*\dx, \yy) circle [radius=2pt];
}

\def\myarr{{1,1,4,4,2,9,3,5,9,7,7,8,8,6}};

\foreach \i in {1,...,14}{
    \draw (\i*\dx+\dx/2, \yy) node [above] { \scriptsize \pgfmathparse{\myarr[\i-1]}\pgfmathresult};
}
\end{tikzpicture}
\end{center}
\caption{Minimum weight spanning path $P$.}
\label{fig3}
\end{figure}
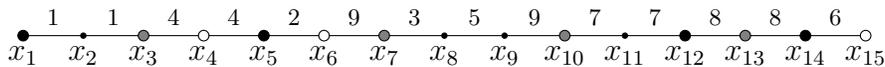

\begin{example}\label{e5}
In this example we are going to apply Theorem~\ref{t44} for calculating Hausdorff distance between subsets of a finite ultrametric space. Let $Z$ be an ultrametric space with the representing tree $T_Z$ depicted in Figure~\ref{fig2} and let
$$
X=\{x_3,x_4,x_6,x_7,x_{10},x_{13},x_{15}\}, 
Y=\{x_1,x_3,x_5,x_7,x_{10},x_{12},x_{13},x_{14}\}.
$$
The points of the subset $X\setminus Y$ are denoted by white circles, of the subset $Y\setminus X$ by black and of the subset $X\cap Y$ by gray.
Since the labeling of the tree $T_Z$ is injective using Proposition~\ref{lbpm} we see that there is one-to-one correspondence between the labels of inner nodes of $T_Z$ and the set of nonsingular balls of the space $Z$. Thus, denote by $B_i$ the ball with the diameter $i$, $i\in \operatorname{Sp}(Z)\setminus \{0\}$.
One can easily see that the following balls satisfy condition (i) of Definition~\ref{d21}: $B_1$, $B_4$, $B_2$, $B_9$, $B_8$, $B_7$, $B_6$.
Among them only the balls $B_1$, $B_4$, $B_2$, $B_7$, $B_6$ satisfy condition (ii) of the same definition. Thus,  $\mathfrak B_{XY}=\{B_1, B_2, B_4, B_6, B_7\}$.  By~(\ref{e45}) we have $d_H(X,Y)=7$.
\end{example}

\begin{remark}
One of the minimum spanning paths $P$ of the space $Z$ is  depicted in Figure~\ref{fig3}. The reader can easily check the correctness of Remark~\ref{r55} and establish the set of balls of the space $Z$ considering the path $P$. Verifying conditions (i) and (ii) of Definition~\ref{d21} using $P$ even is more easy than using $T_Z$. This observation again shows that minimum spanning paths are a convenient tool for studying finite ultrametric spaces.
\end{remark}

\section{Funding}
The research was partially supported by the National Academy of Sciences of Ukraine, Project 0117U002165 ``Development of mathematical models, numerically analytical methods and algorithms for solving modern medico-biological problems''.

\end{document}